\documentclass{amsart}
\usepackage{amsmath}
\usepackage{amssymb}
\usepackage{amsthm}
\usepackage{color}
\usepackage{MnSymbol}
\usepackage{enumitem}

\setlist[enumerate]{font=\normalfont}

\newtheorem{theorem}{Theorem}
\newtheorem{lemma}[theorem]{Lemma}
\newtheorem{proposition}[theorem]{Proposition}
\newtheorem{conjecture}[theorem]{Conjecture}
\newtheorem{corollary}[theorem]{Corollary}

\theoremstyle{definition}
\newtheorem{definition}[theorem]{Definition}

\newtheorem{question}[theorem]{Question}
  
\theoremstyle{remark}

\newcommand{\Z}{\mathbb{Z}}
\newcommand{\Q}{\mathbb{Q}}
\newcommand{\N}{\mathbb{N}}
\newcommand{\A}{\mathcal{A}}

\title{The Conjugacy Growth of the Soluble Baumslag-Solitar Groups}
\author{Laura Ciobanu} 
\address{School of Mathematical and Computer Sciences,
 Heriot-Watt University, 
 Edinburgh EH14 4AS,
 Scotland}
\email{l.ciobanu@hw.ac.uk}

\author{Alex Evetts}
\address{School of Mathematical and Computer Sciences,
 Heriot-Watt University, 
 Edinburgh EH14 4AS,
 Scotland}
\email{ace2@hw.ac.uk}

\author{Meng-Che ``Turbo'' Ho}
\address{Department of Mathematics,
Purdue University,
West Lafayette, IN 47907-2067,
USA
}
\email{ho140@purdue.edu}
\date{\today}

\begin{document}

\maketitle
\begin{abstract}
 In this paper we give asymptotics for the conjugacy growth of the soluble Baumslag-Solitar groups $BS(1,k)$, $k\geq 2$, with respect to the standard generating set, by providing a complete description of geodesic conjugacy representatives. We show that the conjugacy growth series for these groups are transcendental, and give formulas for the series. As a result of our computation we also establish that in each $BS(1,k)$ the conjugacy and standard growth rates are equal. \\
 
 
 \end{abstract}

\section{Introduction}
For any $n\geq 0$, the \textit{conjugacy growth function} $c_{G,S}(n)$ of a 
finitely generated group $G$, with respect to some finite generating set $S$, counts the number of conjugacy classes intersecting the ball of radius $n$ in the Cayley graph of $G$ with respect to $S$. The \emph{conjugacy growth series} of $G$ with respect to $S$ is then the generating function for the sequence $c_{G,S}(n)$. There are numerous results in the literature about the asymptotics of conjugacy growth \cite{ckGAFA, ckIJAC, gubasapir, hullosin}, as well as about the behaviour of conjugacy growth series \cite{AC2017, CH2014, CHHR, Evetts, Mercier, rivin}, for important classes of groups. Of particular relevance here is the work \cite{bdc} of Breuillard and Cornulier, who showed that the function $c_{G,S}(n)$ grows exponentially for finitely generated soluble groups that are not virtually nilpotent, such as the soluble Baumslag-Solitar groups $BS(1,k) = \langle a,t \mid tat^{-1} =  a^k \rangle$, $k\geq 2$.

 In this paper we give finer asymptotics for $c_{BS(1,k), \{a, t\}}(n)$, compute explicitly the conjugacy growth series of $BS(1,k)$ with respect to the standard generating set $\{a,t\}$, and show that this series is transcendental. We establish the transcendental behaviour from the fact that $c_{G,S}(n)$ is asymptotically of the form $\frac{\alpha^n}{n}$ for a constant $\alpha>1$, which is interestingly similar to hyperbolic groups \cite{AC2017} and several classes of acylindrically hyperbolic groups \cite{GY2019}, despite $BS(1,k)$ being among the first examples of groups that are not acylindrically hyperbolic.
 
 This paper provides further confirmation for the conjecture (see \cite{Evetts}) that the only groups with rational conjugacy growth series are the virtually abelian ones. It also provides further confirmation for the conjecture that the conjugacy and standard growth rates in finitely presented groups are equal; this was already observed for hyperbolic \cite{AC2017}, relatively hyperbolic \cite{GY2019}, most graph products \cite{CHM18} and lamplighter groups \cite{Mercier}.
 
 The structure of the paper is as follows. We give the background on conjugacy growth functions and series in Section \ref{sec:prelim}, where we also provide descriptions of normal forms in the Baumslag-Solitar groups that will be used to describe the conjugacy representatives later in the paper. In Section \ref{sec:abelian} we completely describe the conjugacy representatives and give the conjugacy growth series for those elements in the maximal abelian normal subgroup of $BS(1,k)$, and then in Section \ref{sec:general} we describe geodesic conjugacy representatives for the remaining conjugacy classes of $BS(1,k)$. 
 
 The main result appears in Section \ref{sec:growthseries}, where we show (Corollary \ref{cor:mainresult}) that the conjugacy growth series of $BS(1,k)$, with respect to the standard generating set, is transcendental. In Section \ref{sec:growthseries} we also show that the conjugacy and standard growth rates are equal, in Corollary \ref{cor:equalrates}. Finally, in Section \ref{sec:formulas} we give the formulas for the conjugacy growth series of $BS(1,k)$.

\section{Preliminaries} \label{sec:prelim}

\subsection{Conjugacy growth and series} 

Throughout this subsection, fix some group $G$ and a finite generating set $S$ of $G$. The \emph{(word) length} of an element $g\in G$, denoted by $|g|$, is the length of a shortest word in $S$ that represents $g$, i.e.\ $|g| = \min \{ |w| \mid w\in S^*, w =_G g\}$. In this case, we say $w$ is a geodesic word, or simply a geodesic.

We will often write $g \sim h$ to denote that $g$ and $h$ are conjugate, and write $[g]$ for the conjugacy class of $g$. The \emph{length} of $[g]$, denoted by $|[g]|$, is the shortest length among all elements in $[g]$, i.e.\ $|[g]| = \min \{ |h| \mid h \sim g \}$. We say that a word $w$ is a \emph{conjugacy geodesic} for $[g]$ if it a geodesic, and if it moreover represents an element of shortest length in $[g]$.

We define the \emph{cumulative conjugacy growth function} of $G$ with respect to $S$ to be the number of conjugacy classes whose length is at most $n$, and the \emph{strict conjugacy growth function}, denoted as $c(n) = c_{G,S}(n)$, to be the number of conjugacy classes whose length is equal to $n$, i.e.\ 
\[ c(n) = \#\{[g] \mid |[g]| = n\}.\]
For ease of computation we shall work only with the strict version, and call that the \emph{conjugacy growth function}. 
The \emph{conjugacy growth series} $C(z) = C_{G,S}(z)$ is defined to be the (ordinary) generating function of $c(n)$, so
\[ C(z) = \sum\limits_{n = 0}^{\infty} c(n) z^n. \]
All results in this paper can be easily extended to the cumulative version of the conjugacy growth function and series (see \cite{AC2017}).

We call a formal power series $f(z)$ \emph{rational} if it can be expressed (formally) as the ratio of two polynomials with integral coefficients, or equivalently, the coefficients of $f(z)$ satisfy a finite linear recursion. In the language of polynomial rings, this is to say $f(z) \in \Q(z)$. Furthermore, $f(z)$ is \emph{irrational} if it is not rational. 

A formal power series is \emph{algebraic} if it is in the algebraic closure of $\Q(z)$, i.e.\ it is the solution to an polynomial equation with coefficients from $\Q(z)$. It is called \emph{transcendental} if it is not algebraic.

\subsection{Baumslag-Solitar groups.} Throughout the rest of the article, we will write $$G = BS(1,k) = \langle a,t \mid tat^{-1} =  a^k \rangle$$ where $k \geq 2$ is a natural number, and will write the conjugation as $a^t = tat^{-1}$. Let $\Z_k = \{x \in \Q \mid k^ex \in \Z \text{ for some }e\in \Z\}$ and consider the semidirect product $\Z_k\rtimes \Z$, where the action of $\Z$ on $\Z_k$ is multiplication by $k$. Then $BS(1,k) \cong \Z_k\rtimes \Z$, with the isomorphism given by $a \to (1,0) \in \Z_k$ and $t\to (0,1) \in \Z$ where we write an element of $G$ in the semidirect normal form $(x,m)$. 

Suppose that $m > 0$. Since 
\begin{equation}
(t^m)^a = at^ma^{-1} = a\cdot a^{-k^m}t^m = (1-k^m,m)
\end{equation}
and $a^t = a^k$, we get that conjugation by generators amounts to:
\begin{equation} \label{basicconjugation}
(x,m)^a = (x+(1-k^m),m)
\  \textrm{and} \ (x,m)^t = (kx,m).
 \end{equation} 

The form of geodesics in the soluble Baumslag-Solitar groups has been studied in several articles, and we summarise here the results in a form convenient for further use. The following propositions are derived from section $4$ of \cite{CEG}. We restrict for now to only those elements with zero $t$-exponent sum. 

\begin{proposition}\label{prop:basegrpelemreps}
	 Let $k=2r+1$ for some positive integer $r$. The set $\mathcal{E}_o$ of words in the following forms comprises a set of unique geodesic representatives for the elements of the subgroup $\Z_k$.
	\begin{enumerate}[label=\normalfont O\alph*.]
		\item\label{odd1} $\{\epsilon,a^{\pm1},\ldots a^{\pm(r+1)}\} $
		\item\label{odd2} $\{a^{x_0}ta^{x_1}\cdots ta^{x_d}t^{-d}\mid d\geq1, x_d\neq0, A\}$
		\item\label{odd3} $\{t^{-b}a^{x_0}ta^{x_1}\cdots ta^{x_d}t^{-c}\mid b,c,d\geq1, b=c+d, x_0\neq0, x_d\neq0, A \}$
		\item\label{odd4} $\{t^{-d}a^{x_0}ta^{x_1}\cdots ta^{x_d} \mid d\geq1, x_0\neq0, A  \}$
	\end{enumerate}
Here $A$ signifies the conditions $|x_d|\leq r+1$, $|x_i|\leq r$ for $i<d$, and if $x_{d-1}=\pm r$ then $x_d\neq\mp1$.
\end{proposition}


\begin{proposition}\label{prop:evenelemreps}
Let $k=2r$ for some $r\geq 2$. The set $\mathcal{E}_e$ of words in the following forms comprise a set of unique geodesic representatives for the elements in $\Z_k$.
\begin{enumerate}[label=\normalfont E\alph*.]
\item\label{even1} $\{\epsilon,a^{\pm1},\ldots,a^{\pm(r+1)}\}$ 
\item\label{even2} $\{a^{x_0}ta^{x_1}\cdots ta^{x_d}t^{-d}\mid  d\geq1, x_d \neq 0, A, B \}$
\item\label{even3} $\{t^{-b}a^{x_0}ta^{x_1}\cdots ta^{x_d}t^{-c}\mid b,c,d\geq1, b=c+d, x_0 \neq 0, x_d \neq 0, A, B \}$
\item\label{even4} $\{t^{-d}a^{x_0}ta^{x_1}\cdots ta^{x_d} \mid d\geq1, x_0 \neq 0, A \}$
\end{enumerate}
Here, $A$ signifies the conditions $|x_d|\leq r+1$, and for each $0\leq i<d$, $|x_i|\leq r$, if $x_{i-1}=r$ then $0\le x_i<r$ for $i<d$, and if $x_{i-1}=-r$ then $-r<x_i\le 0$. And $B$ signifies that the following subwords are forbidden: $a^{\pm r}ta^{\pm(r-2)}ta^{\mp1}t^{-1}$, $a^{\pm(r-1)}ta^{\mp1}t^{-1}$.
\end{proposition}


\begin{proposition}\label{prop:2elemreps}
	 Let $k=2$, i.e. $G=BS(1,2)$. The set $\mathcal{E}_2$ of words in the following forms comprise a set of unique geodesic representatives for the elements in $\Z_k$.
	\begin{enumerate}[label=\normalfont 2\alph*.]
		\item\label{two-1} $\{\epsilon,a^{\pm1},a^{\pm2},a^{\pm3}\}$
		\item\label{two-2} $\{a^{x_0}ta^{x_1}t\cdots ta^{x_d}t^{-d}\mid d\geq1, |x_d|\in\{2,3\}, A\}$
		\item\label{two3} $\{t^{-b}a^{x_0}ta^{x_1}\cdots ta^{x_d}t^{-c}\mid b,c,d\geq1, d=b+c, x_0 \neq 0, |x_d|\in\{2,3\}, A\}$
		\item\label{two4} $\{t^{-d}a^{x_0}t\cdots ta^{x_d}\mid d\geq1, x_0 \neq 0, A\}$
	\end{enumerate}
Here, $A$ signifies the conditions $|x_i|\leq1$ for $i<d$, if $x_{i-1}\neq0$ then $x_i=0$ for $i<d$, if $x_d>0$ then $x_{d-1}\geq0$, and if $x_d<0$ then $x_{d-1}\leq 0$.
\end{proposition}

\subsection{Context-free languages}
We will need some formal language theory (see for example \cite{HU}) in order to calculate the growth series of $\Z_k$ in Section \ref{sec:abelian}.
\begin{definition}
	Let $\mathcal{V}$ be a set of variables (usually denoted by upper case letters), and $\mathcal{T}$ a set of terminals (usually denoted by lower case letters).
	A \emph{context-free grammar} consists of a finite set of production rules of the form \[V\rightarrow w_1\mid w_2\mid\cdots\mid w_n\]
	where $V\in\mathcal{V}$, each $w_i\in(\mathcal{V}\cup\mathcal{T})^*$, and the $\mid$ symbol stands for exclusive `or'. We nominate one variable to be the starting variable.
\end{definition}
A context-free grammar produces a language in the following way. Start at the nominated starting variable, and perform substitutions according to the production rules, until the word consists only of terminals. The language $L\subset\mathcal{T}^*$ of all words that can be produced from the grammar is called a \emph{context-free} language. If each word is only produced in one way (i.e. via a unique sequence of production rules) then the language is called \emph{unambiguous context-free}.

\begin{theorem}[Chomsky-Sch\"utzenberger \cite{CS}] \label{CS}
	If $L$ is an unambiguous context-free language, its growth series is algebraic.
\end{theorem}

There is a method for explicitly calculating the series, known as the DSV method, which is as follows. Convert the grammar into a system of equations by replacing:
\begin{itemize}
	\item the empty word $\epsilon$ with the integer $1$,
	\item each terminal with the formal variable $z$,
	\item each variable $V$ with a function $V(z)$,
	\item the or $|$ with addition $+$,
	\item concatenation with multiplication,
	\item the production arrow with $=$.
\end{itemize}
Solving the system of equations for the initial variable then gives the growth series, an algebraic function of $z$.

\section{The conjugacy classes $[(x,0)]$ in $BS(1,k)$}\label{sec:abelian}

In this section we show that the conjugacy growth series of the subgroup $\Z_k$, relative to $G=BS(1,k)$, is rational with respect to the generating set $\{a,t\}$. We explicitly calculate the series, and extract the growth rate.

\begin{lemma}
Two elements $(x,0)$ and $(y,0)$ are conjugate in $G$ if and only if there is some $e\in \Z$ such that $x = k^e y$.
\end{lemma}
\begin{proof}
The group $\Z_k$ is abelian, and $(y,0)\in\Z_k\lhd\Z_k\rtimes\Z$, thus $(y,0)^a=(y,0)$. We also have $(y,0)^t=(ky,0)$. Therefore
$[(y,0)]=(y,0)^{\langle t\rangle} = \{(k^ey,0)\mid e\in\Z\}$.
\end{proof}

Thus, for every conjugacy class $[(y,0)]$, there is a unique $(x,0)\in[(y,0)]$ such that $x\in\Z$ and $k\nmid x$.

We treat the cases $k$ odd and even separately.

\subsection{Odd case}

Let $k=2r+1$ for some integer $r\geq 1$. 
\begin{proposition}
	In $BS(1, 2r+1)$ the set of words \[\mathcal{C}_o=\{\epsilon,a^{\pm1},\ldots a^{\pm(r+1)}\}\cup \{a^{x_0}ta^{x_1}t\cdots ta^{x_d}t^{-d}\mid d\geq1, x_0\neq0, x_d\neq0, A\}, \] where $A$ signifies the conditions $|x_d|\leq r+1$, $|x_i|\leq r$ for $i<d$, and if $x_{d-1}=\pm r$ then $x_d\neq\mp1$, comprises a set of unique geodesic representatives for the conjugacy classes of $G$ that lie in $\Z_k$.
\end{proposition}

\begin{proof}
Let $\mathcal{E}_o$ be as in Proposition  \ref{prop:basegrpelemreps} and note that $\mathcal{C}_o\subset\mathcal{E}_o$. We use the following key observation: if an element is represented by a word in $\mathcal{E}_o\setminus\mathcal{C}_o$, then it cannot be represented by a word in $\mathcal{C}_o$, by the uniqueness condition on $\mathcal{E}_o$. We will first prove that no pair of words in $\mathcal{C}_o$ represent the same conjugacy class, and then prove that every word in $\mathcal{E}_o$ is conjugate to a word in $\mathcal{C}_o$ with at most the same length. Then since every group element is represented in $\mathcal{E}_o$, every conjugacy class is represented (uniquely) in $\mathcal{C}_o$. Furthermore, this unique representative has length at most that of each of the corresponding (element-minimal) representatives in $\mathcal{E}_o$. This proves the proposition.
	
	Proposition \ref{prop:basegrpelemreps} implies that no two words in $\mathcal{C}_o$ represent equal elements. We show that no two words represent conjugate group elements either. Suppose, on the contrary, that $w,v\in\mathcal{C}_o$ represent conjugate elements. So there exists a non-zero integer $m$ such that $t^mwt^{-m}=_G v$. First suppose that $w=a^n$ for $|n|\leq r+1$. Then $t^ma^nt^{m-1}$ is a word in either \eqref{odd2} (with $x_0=0$) or \eqref{odd3}, depending on the sign of $m$, and thus by the above observation the word $v\notin\mathcal{C}_o$, which is a contradiction. Now suppose that $w=a^{x_0}ta^{x_1}t\cdots a^{x_d}t^{-d}$ for $d\geq 1$, $x_0\neq0$, with conditions $A$ and $B$. So $v=t^mwt^{-m}=t^ma^{x_0}ta^{x_1}t\cdots a^{x_d}t^{-d-m}$. If $m>0$, $v$ is a word in \eqref{odd2} (and not in $\mathcal{C}_o$). If $m<0$, $v$ is a word in either \eqref{odd3} or \eqref{odd4}. In both cases $v\notin\mathcal{C}_o$, which is again a contradiction.
	
	Now let $w\in\mathcal{E}_o$. We show that there exists $v\in\mathcal{C}_o$ such that $w$ and $v$ represent conjugate group elements, and moreover $|w|\geq|v|$ (as words). We assume that $w\notin\mathcal{C}_o$ (otherwise the claim is trivial). First, suppose that $w$ is in form $\eqref{odd2}$, and let $i>0$ be such that $x_i$ is the left-most non-zero power of $a$. Then the word $v=a^{x_i}ta^{x_{i+1}}t\cdots ta^{x_d}t^{-d+i}$ is in $\mathcal{C}_o$ and represents a conjugate of $\overline{w}$. Further, the number of $a^{\pm1}$s in $v$ is the same as that in $w$, and the number of $t^{\pm1}$s in $v$ is $(d-i)+(d-i)<2d$ and therefore $|v|<|w|$. Now suppose $w$ is of the form \ref{odd3} (resp. \ref{odd4}). Let $v=a^{x_0}ta^{x_1}t\ldots a^{x_d}t^{-d}$. Since $v$ is a leftward cyclic permutation of $w$ by $b$ (resp. $d$) places, the words represent conjugate elements and are of equal length.	
\end{proof}
\begin{proposition}\label{prop:abelianodd} Let $k=2r+1$, where $r \geq 1$. 
\begin{enumerate}
		\item\label{two1} In $BS(1, k)$ the set $\mathcal{C}_o$ is unambiguous context-free.
		\item\label{two2} The subgroup $\Z_k$ has rational relative conjugacy growth.
\end{enumerate}
\end{proposition}
\begin{proof}
(1)
First note that $\mathcal{C}_o$ is not regular since the exponent-sum of $t$ has to be $0$, and this cannot be achieved by a finite state automaton.

We show the language is context-free by exhibiting an explicit grammar. We use capital letters for variables and lower case for terminals. 
Write $a^{\pm n}$ as shorthand for the concatenation of $n$ copies of the terminal $a^{\pm1}$. It is easy to see that the following context-free grammar, starting from $S$, produces the set in question unambiguously.

\begin{align*} 
S&\rightarrow\epsilon ~|~ A ~|~ T,  \  \  A\rightarrow a^{-r-1}~|~\cdots~|~a^{-1}~|~a~|~\cdots~|~a^{r+1}\\
B&\rightarrow a^{-r+1}~|~\cdots ~|~ a^{-1}~|~a~|~\cdots~|~a^{r-1}, \ \ T\rightarrow BtUt^{-1}~|~a^r tVt^{-1}~|~a^{-r}tWt^{-1}\\
U&\rightarrow A~|~tUt^{-1}~|~T, \ \ V\rightarrow tUt^{-1}~|~T~|~a^{-r-1}~|~\cdots~|~a^{-2}~|~a~|~\cdots~|~a^{r+1}\\
W&\rightarrow tUt^{-1}~|~T~|~a^{-r-1}~|~\cdots~|~a^{-1}~|~a^2~|~\cdots~|~a^{r+1}.
\end{align*}

(2) By Theorem \ref{CS} the growth series of the language $\mathcal{C}_o$, and hence the relative conjugacy growth series of the subgroup $\Z_k$, is algebraic. However, a stronger result holds here. 
%
Applying the DSV method to the grammar above gives the growth series of the language $\mathcal{C}_o$. The production rules become the equations:
\begin{align*}
S(z)&=1 + A(z) + T(z), \ \ A(z)=2\sum_{i=1}^{r+1}z^i = 2\frac{z-z^{r+2}}{1-z},\\
B(z)&=2\sum_{i=1}^{r-1}z^i = 2\frac{z-z^r}{1-z}, \ \ T(z)=B(z)U(z)z^2 + V(z)z^{r+2} + W(z)z^{r+2},\\
U(z)&=A(z) + U(z)z^2 + T(z), \ \ V(z)=U(z)z^2 + T(z) + 2\sum_{i=1}^{r+1}z^i - z,\\
W(z)&=U(z)z^2 + T(z) + 2\sum_{i=1}^{r+1}z^i - z.
\end{align*}
Solving these equations for $S(z)$ we find that
\begin{equation}\label{abelianseriesodd}
S(z) = \frac{2z^{r+6}-2z^{r+5}-4z^{r+4}+2z^{r+2}+3z^3+z^2-z-1}{z^3-2z^{r+3}+z^2+z-1}.
\end{equation}
\end{proof}

\begin{corollary}\label{cor:oddabelian}
The conjugacy classes in $\Z_k$, for $k=2r+1$, have growth rate in the range $(\frac{4}{3},2)$. 
\end{corollary}
\begin{proof}
Denote by $d_o$ the denominator of $S(z)$ in (\ref{abelianseriesodd}), that is, $d_o(z)=z^3-2z^{r+3}+z^2+z-1= z^3(1-z^r)+z(1-z^{r+2})+(z^2-1)$, which implies that for $z \in (-1,0)$, $d_o(z) < 0$. Also, $d_o(\frac{1}{2})=-\frac{1}{8}-\frac{1}{2^{r+2}}<0$ and $d_o(\frac{3}{4})=\frac{47}{64}-\frac{27}{32}(\frac{3}{4})^r>0$, so there is a smallest root $\rho_o \in (\frac{1}{2}, \frac{3}{4})$ of $d_o$. Furthermore, $d_o(0)=-1$ and $d'_o(z)>0$ for $z\in[0,\frac{1}{2}]$, so $\rho_o$ is the real root with smallest absolute value. 

Write $a=\rho_0$ for ease of notation. The fact that $a$ is a root of the denominator gives $2a^{r+3}=a^3+a^2+a-1$. Using this identity we can substitute each $a^{\geq r}$ by the appropriate expression into the numerator and obtain $2a^{r+6}-2a^{r+5}-4a^{r+4}+2a^{r+2}+3a^3+a^2-a-1 =$ $a^7-2a^5-a^4+a^3+2a^2-1$. Furthermore, $a^7-2a^5-a^4+a^3+2a^2-1=0$ only for $a=\pm 1$, which is not the case, as $a \in (\frac{1}{2}, \frac{3}{4})$. Thus  $\rho_o$ is not a root of the numerator of $S(z)$ in (\ref{abelianseriesodd}), so the growth rate, which is the reciprocal of $\rho_o$, lies in the given range.
\end{proof}

\subsection{Even case}

Let $k=2r$, for some integer $r\geq2$.

\begin{proposition}
In $G=BS(1,2r)$, $r\geq 2$, the set of words 
\begin{align*}
\mathcal{C}_e=\{\epsilon,a^{\pm1},\ldots,a^{\pm(r+1)}\}\cup \{a^{x_0}ta^{x_1}\cdots a^{x_d}t^{-d}\mid d\geq1,A,B,x_0\neq0\}
\end{align*}
comprises a set of unique geodesic representatives for the conjugacy classes of $G$ that lie in $\Z_k$. Here, $A$ signifies the conditions $|x_d|\leq r+1$, and for each $0\leq i<d$, $|x_i|\leq r$, if $x_{i-1}=r$ then $0<x_i<r$ for $i<d$, and if $x_{i-1}=-r$ then $-r<x_i<0$. And $B$ signifies that the following subwords are forbidden: $a^{\pm r}ta^{\pm(r-2)}ta^{\mp1}t^{-1}$, $a^{\pm(r-1)}ta^{\mp1}t^{-1}$.
\end{proposition}

\begin{proof}
Let $\mathcal{E}_e$ be as in Proposition \ref{prop:evenelemreps} and note that $\mathcal{C}_e\subset\mathcal{E}_e$. We use the following key observation: if an element is represented by a word in $\mathcal{E}_e\setminus\mathcal{C}_e$, then it cannot be represented by a word in $\mathcal{C}_e$, by the uniqueness condition on $\mathcal{E}_e$. We will first prove that no pair of words in $\mathcal{C}_e$ represent the same conjugacy class, and then prove that every word in $\mathcal{E}_e$ is conjugate to a word in $\mathcal{C}_e$ with at most the same length. Then since every group element is represented in $\mathcal{E}_e$, every conjugacy class is represented (uniquely) in $\mathcal{C}_e$. Furthermore, this unique representative has length at most that of each of the corresponding (element-minimal) representatives in $\mathcal{E}_e$. This proves the proposition.

Proposition \ref{prop:evenelemreps} implies that no two words in $\mathcal{C}_e$ represent equal elements. We show that no two words represent conjugate group elements either. Suppose, on the contrary, that $w,v\in\mathcal{C}_e$ represent conjugate elements. So there exists a non-zero integer $m$ such that $t^mwt^{-m}=_G v$. First suppose that $w=a^n$ for $|n|\leq r+1$. Then $t^ma^nt^{m-1}$ is a word in either \eqref{even2} (with $x_0=0$) or \eqref{even3}, depending on the sign of $m$, and thus by the above observation the word $v\notin\mathcal{C}_e$, which is a contradiction. Now suppose that $w=a^{x_0}ta^{x_1}t\cdots a^{x_d}t^{-d}$ for $d\geq 1$, $x_0\neq0$, with conditions $A$ and $B$. So $v=t^mwt^{-m}=t^ma^{x_0}ta^{x_1}t\cdots a^{x_d}t^{-d-m}$. If $m>0$, $v$ is a word in \eqref{even2} (and not in $\mathcal{C}_e$). If $m<0$, $v$ is a word in either \eqref{even3} or \eqref{even4}. In both cases, we have $v\notin\mathcal{C}_e$, which is again a contradiction.

Now let $w\in\mathcal{E}_e$. We claim that there exists $v\in\mathcal{C}_e$ such that $w$ and $v$ represent conjugate group elements, and moreover $|w|\geq|v|$ (as words). We assume that $w\notin\mathcal{C}_e$ (otherwise the claim is trivial).

There are two exceptional cases. First, suppose $w=t^{-d}a^{x_0}ta^{x_1}\cdots a^{x_{d-2}}ta^{\pm(r-1)}ta^{\mp1}$ with $d\geq1$ and conditions $A$ (so in particular $w$ is in the form \eqref{even4}). Then $\overline{w}$ is conjugate to the element represented by $a^{x_0}ta^{x_1}\cdots a^{x_{d-1}}ta^{\pm(r-1)}ta^{\mp1}t^{-d}$. This word contains a forbidden subword and therefore does not satisfy condition $B$, so is not in $\mathcal{C}_e$. However, it represents the same element as $v:=a^{x_0}ta^{x_1}\cdots a^{x_{d-1}}ta^{\mp(r+1)}t^{-d+1}\in\mathcal{C}_e$. We also have $|w|=\sum_{i=0}^{d-2}x_i + (r-1)+1+2d>\sum_{i=0}^{d-2}x_i+(r+1)+2(d-1)=|v|.$

For the second exceptional case, suppose $w=t^{-d}a^{x_0}ta^{x_1}\cdots a^{x_{d-3}}ta^{\pm r}ta^{\pm(r-2)}ta^{\mp1}.$ Then $\overline{w}$ is conjugate to the element represented by $a^{x_0}ta^{x_1}\cdots a^{x_{d-3}}ta^{\pm r}ta^{\pm(r-2)}ta^{\mp1}t^{-d}$, which contains a forbidden subword, but represents the same element as $v:=a^{x_0}ta^{x_1}\cdots a^{x_{d-3}}ta^{\mp r}ta^{\mp(r+1)}t^{-d+1}\in\mathcal{C}_e$. In this case we have \[|w|=\sum_{i=0}^{d-3}x_i + r+(r-2)+1+2d=\sum_{i=0}^{d-3}x_i+r+(r+1)+2(d-1)=|v|.\]

For the general case, where $w$ is in the form \eqref{even2}, \eqref{even3}, or \eqref{even4} (excluding the exceptional cases) and is not already an element of $\mathcal{C}_e$, it is clear that conjugation by $t^{\pm1}$ an appropriate number of times takes $w$ to a word in $\mathcal{C}_e$, which has at most the same length as $w$.
\end{proof}

\begin{proposition}\label{prop:abelianeven} Let $k=2r$, $r\geq 2$.
\begin{enumerate}
		\item\label{two1}
In $G=BS(1,k)$, the set $\mathcal{C}_e$ is an unambiguous context-free language.
\item\label{two2}The subgroup $\Z_k$ has rational conjugacy growth.
\end{enumerate}
\end{proposition}
\begin{proof}
(1) We claim that the following grammar, with $S$ as the starting point, generates $\mathcal{C}_e$ unambiguously. 
\begin{align*}
S&\rightarrow\epsilon\mid A\mid T, \ \ A\rightarrow a^{-(r+1)}\mid a^{-r}\mid\cdots\mid a^{-1}\mid a\mid\cdots\mid a^{r+1}\\
T&\rightarrow BtUt^{-1}\mid a^rtVt^{-1} \mid a^{-r}tWt^{-1}\mid a^{r-1}tXt^{-1}\mid a^{-(r-1)}tYt^{-1}\\
B&\rightarrow a^{-(r-2)}\mid a^{-(r-3)}\mid \cdots \mid a^{-1}\mid a\mid \cdots\mid a^{r-2}, \ \ U\rightarrow tUt^{-1}\mid T\\
V&\rightarrow a\mid a^2\mid \cdots\mid a^{r-1}\mid tUt^{-1}\mid atUt^{-1}\mid \cdots\mid a^{r-3}tUt^{-1}\mid a^{r-2}tXt^{-1}\mid a^{r-1}tXt^{-1}\\
W&\rightarrow a^{-1}\mid a^{-2}\mid\cdots\mid a^{-(r-1)}\mid tUt^{-1}\mid a^{-1}tUt^{-1}\mid\cdots\mid a^{-(r-3)}tUt^{-1}\mid a^{-(r-2)}tYt^{-1}\mid a^{-(r-1)}tYt^{-1}\\
X&\rightarrow a^{-(r+1)}\mid a^{-r}\mid\cdots\mid a^{-2}\mid a\mid a^2\mid\cdots\mid a^{r+1}\mid U\\
Y&\rightarrow a^{-(r+1)}\mid a^{-r}\mid\cdots\mid a^{-2}\mid a^{-1}\mid a^2\mid\cdots\mid a^{r+1}\mid U
\end{align*}
Starting from $S$, this grammar produces words in $\mathcal{C}_e$ by choosing the values of the powers $x_i$ from left to right, while keeping track of the number $d$ of such powers. If $x_i$ is chosen to be $\pm r$ or $\pm(r-1)$, restrictions apply to the following power. 
%

(2) We use the grammar above to explicitly calculate the growth function. The grammar yields the following system of equations.
\begin{align*}
S(z)&=1+A(z)+T(z), \ \ A(z)=2\sum_{i=1}^{r+1}z^i = \frac{2(z-z^{r+2})}{1-z},\\
T(z)&=t^2B(z)U(z) + 2z^{r+2}V(z) + 2z^{r+1}X(z),\\
B(z)&=2\sum_{i=1}^{r-2}z^i = \frac{2(z-z^{r-1})}{1-z}, \ \ U(z)= z^2U(z) + T(z),\\
V(z)=W(z)&=\sum_{i=1}^{r-1}z^i + z^2U(z)\sum_{i=0}^{r-3}z^i + z^2X(z)(z^{r-2}+z^{r-1})\\
&=\frac{z-z^r}{1-z} + z^2U(z)\frac{1-z^{r-2}}{1-z} + z^2X(z)(z^{r-2}+z^{r-1}),\\
X(z)=Y(z)&=2\sum_{i=1}^{r+1}z^i-z+U(z)=\frac{2(z-z^{r+2})}{1-z}-U(z).
\end{align*}
Solving these for $S(z)$ yields the following rational expression: 
{\footnotesize
\begin{align}\label{abelianserieseven}
S(z)=\frac{n(z)}{d(z)}&=\frac{-1-2z^{r+2}+2z^3+z^4+2z^2-4z^{3r+6}+4z^{3r+8}-2z^{2r+8}+4z^{3r+4}-4z^{r+6}}{(2z^{2r+4}-2z^{r+4}-z^3+2z^{r+2}-z^2-z+1)(z-1)}\\
\nonumber&+\frac{4z^{6+2r}-2z^{2r+7}+2z^{2r+2}+2z^{r+5}-6z^{2r+4}-6z^{r+3}+6z^{r+4}}{(2z^{2r+4}-2z^{r+4}-z^3+2z^{r+2}-z^2-z+1)(z-1)}.
\end{align}}
That is, the denominator of $S(z)$ is $d(z)=(2z^{2r+4}-2z^{r+4}-z^3+2z^{r+2}-z^2-z+1)(z-1)$ and the numerator $n(z)=-1+2z^2+2z^3+z^4-2z^{r+2}-6z^{r+3}+6z^{r+4}+2z^{r+5}-4z^{r+6}
+2z^{2r+2}-6z^{2r+4}+4z^{6+2r}-2z^{2r+7}-2z^{2r+8}+4z^{3r+4}-4z^{3r+6}+4z^{3r+8}$.
\end{proof}

\begin{corollary}\label{cor:evenabelian}
The conjugacy classes in $\Z_k$ have growth rate in the range $(\frac{4}{3},2)$. 
\end{corollary}
\begin{proof}
For $z\in [-\frac12,0]$, $d(z)=2z^{2r+4}-2z^{r+4}-z^3+2z^{r+2}-z^2-z+1 = (1-z) -z^2(1-z^{2r+2})-z^3(1-z^{2r+1})+2z^{r+2}(1-z^2) \ge 1-\frac14-\frac1{16}> 0$, so there is no root in $[-\frac12,0]$. Similarly, for $z\in [-\frac34,-\frac12]$, we have that $d(z) = (1-z) -z^2(1-z^{2r+2})-z^3(1-z^{2r+1})+2z^{r+2}(1-z^2) \ge \frac32-\frac49-2(\frac34)^4> 0$, so there is no root in $[-\frac34,-\frac12]$.

We also have $d(\frac{1}{2})=\frac{1}{8}+\frac{1}{2^{2r+3}}-\frac{1}{2^{r+3}}+\frac{1}{2^{r+1}}>0$ and $d(\frac{3}{4})<0$. So there is a root $\in (\frac{1}{2}, \frac{3}{4})$ of $d$. Furthermore, $d(0)=1$ and $d'(z)<0$ for $z\in[0,\frac{1}{2}]$, so the real root with smallest absolute value lies in $(\frac{1}{2}, \frac{3}{4})$. 

Write $a$ to be the real root with smallest absolute value of $d(z)$. The fact that $a$ is a root of the denominator gives $2a^{2r+4}-2a^{r+4}+2a^{r+2}-a^3-a^2-a+1= 0$. In particular, $2a^{2r+4}= 2a^{r+4}-2a^{r+2}+a^3+a^2+a-1= 0$ and $a^3+a^2+a-1= 2a^{2r+4}-2a^{r+4}+2a^{r+2}$. Using these identities we get that 
{\footnotesize
\begin{align*}
n(a)&=-1+2a^2+2a^3+a^4-2a^{r+2}-6a^{r+3}+6a^{r+4}+2a^{r+5}-4a^{r+6+2a^{2r+2}}\\
&-6a^{2r+4}+4a^{6+2r}-2a^{2r+7}-2a^{2r+8}+4a^{3r+4}-4a^{3r+6}+4a^{3r+8}\\
= &(a+1)(a^3+a^2+a-1)-2a^{r+2}-6a^{r+3}+6a^{r+4}+2a^{r+5}-4a^{r+6}+2a^{2r+2}\\
&-6a^{2r+4}+4a^{6+2r}-2a^{2r+7}-2a^{2r+8}+2a^{2r+4}(2a^r-2a^{r+2}+2a^{r+4})\\
= & (a+1)(2a^{2r+4}-2a^{r+4}+2a^{r+2})-2a^{r+2}-6a^{r+3}+6a^{r+4}+2a^{r+5}-4a^{r+6}+2a^{2r+2}\\
&-6a^{2r+4}+4a^{6+2r}-2a^{2r+7}-2a^{2r+8}+(2a^{r+4}-2a^{r+2}+a^3+a^2+a-1)(2a^r-2a^{r+2}+2a^{r+4})\\
= & 2a^r(a+1)(a-1)^2((a^3-a-1)a^{r+2}+a^4+a^2-1)
\end{align*}}
However, $a\in (\frac12,\frac34)$ implies $a^3-a-1 < 0$ and $a^4+a^2-1 < 0$, so $((a^3-a-1)a^{r+2}+a^4+a^2-1) < 0$. Also $a \neq -1,0,1$, so $a$ is not a root of the numerator of $S(z)$ in (\ref{abelianserieseven}), and thus the growth rate, which is the reciprocal of $a$, lies in the given range.

\end{proof}

\subsection{The case $k=2$}
Let $G=BS(1,2)$.

\begin{proposition}
	In $BS(1,2)$ the set of words
	\begin{equation*}
	\mathcal{C}_2=\{\epsilon,a^{\pm1},a^{\pm3}\}\cup\{a^{x_0}ta^{x_1}t\cdots ta^{x_d}t^{-d}\mid d\geq1, |x_d|\in\{2,3\}, x_0\neq0, A\}
	\end{equation*}
	comprises a set of unique geodesic representatives for the conjugacy classes of $G$ that lie in the subgroup $\Z_k$. Here, $A$ signifies the conditions $|x_i|\leq1$ for $i<d$, if $x_{i-1}\neq0$ then $x_i=0$ for $i<d$, if $x_d>0$ then $x_{d-1}\geq0$, and if $x_d<0$ then $x_{d-1}\leq 0$.
\end{proposition}

\begin{proof}
	Let $\mathcal{E}_2$ be as in Proposition \ref{prop:2elemreps} and note that $\mathcal{C}\subset\mathcal{E}_2$. As above, we use the following key observation: if an element is represented by a word in $\mathcal{E}_2\setminus\mathcal{C}_2$, then it cannot be represented by a word in $\mathcal{C}_2$, by the uniqueness condition on $\mathcal{E}_2$. We will first prove that no pair of words in $\mathcal{C}_2$ can represent the same conjugacy class, and then prove that every word in $\mathcal{E}_2$ is conjugate to a word in $\mathcal{C}_2$ of at most the same length, proving the proposition.
	
	We show that no pair of words in $\mathcal{C}_2$ represent conjugate elements. Let $w\in\mathcal{C}_2$ and suppose on the contrary that it represents the same conjugacy class as some $v\in\mathcal{C}_2$. Since no pair of words in $\mathcal{E}_2$ represent the same element, there exists $m\neq0$ with $t^mwt^{-m}=_Gv$. First consider the case where $w\in \{a^{x_0}ta^{x_1}t\cdots ta^{x_d}t^{-d}\mid d\geq1, |x_d|\in\{2,3\}, x_0\neq0, A\}$. Then $t^mwt^{-m}$ has the form \eqref{two-2} with $x_0\neq0$, or \eqref{two3}, or \eqref{two4}, which contradicts the key observation. Now consider the case $w=a^{\pm1}$, with $m=1$. Then $twt^{-1}=ta^{\pm1}t^{-1}=_Ga^{\pm2}$, and hence the word $twt^{-1}$ cannot be in $\mathcal{C}_2$ by the uniqueness condition on $\mathcal{E}_2$. In the case $w=a^{\pm1}$, with $m=-1$, we have $t^{-1}wt$ in the form \eqref{two4}, again a contradiction. Next, consider $w=a^{\pm1}$ with $|m|\geq2$. We have $t^ma^{\pm1}t^{-m}=_G t^{m-1}a^{\pm2}t^{-(m-1)}$, which is a word in the form \eqref{two-2} with $x_0\neq0$, or \eqref{two3}, or \eqref{two4}, again contradicting the key observation. Finally consider the case $w=a^{\pm3}$. Then $t^ma^{\pm3}t^{-m}$ is in the form \eqref{two-2} with $x_0\neq0$, or \eqref{two3}, or \eqref{two4}, again contradicting the key observation.
	
	Now let $w\in\mathcal{E}_2$. We show that there exists $v\in\mathcal{C}_2$ such that $w$ and $v$ represent conjugate elements, and that $|w|\geq|v|$. We assume $w\notin\mathcal{C}_2$. Firstly, $a^{\pm2}$ is conjugate to $a^{\pm1}\in\mathcal{C}_2$, which has strictly shorter length. Now suppose $w=t^{-d}a^{x_0}t\cdots a^{x_{d-1}}ta\in\eqref{two4}$, where $x_{d-1}\in\{0,1\}$. Then $w$ is conjugate, via $t^d$, to $a^{x_0}t\cdots ta^{x_{d-1}}tat^d$, which has the same length. This word represents the same element as $a^{x_0}t\cdots ta^{x_{d-1}+2}t^{d-1}\in\mathcal{C}_2$ which has strictly smaller length. Similarly, if $w=t^{-d}a^{x_0}t\cdots a^{x_{d-1}}ta^{-1}\in\eqref{two4}$, we must have $x_{d-1}\in\{-1,0\}$, and $w$ represents the same conjugacy class as the shorter word $a^{x_0}t\cdots ta^{x_{d-1}-2}t^{d-1}\in\mathcal{C}_2$. In all other cases, $w$ is clearly conjugate, via an appropriate number of $t^{\pm1}$s, to a word in $\mathcal{C}_2$ of equal or shorter length.
\end{proof}

\begin{proposition}\label{prop:abelian2}
	The subgroup $\Z_k$ in $BS(1,2)$ has rational conjugacy growth.	
\end{proposition}
\begin{proof}
It is straightforward to see that the following grammar, starting from $S$, produces $\mathcal{C}_2$ unambiguously.
	\begin{align*}
	S&\rightarrow\epsilon\mid A\mid T, \ \ A\rightarrow a^{-3}\mid a^{-1}\mid a\mid a^3,\\
	T&\rightarrow at^2Ut^{-2}\mid a^{-1}t^2Ut^{-2}\mid ata^2t^{-1}\mid ata^3t^{-1}\mid a^{-1}ta^2t^{-1}\mid a^{-1}ta^3t^{-1},\\
	U&\rightarrow tUt^{-1}\mid T\mid a^{-3}\mid a^{-2}\mid a^2\mid a^3
	\end{align*}
 The grammar becomes the following system of equations.
	\begin{align*}
	S(z)&=1+A(z)+T(z), \ \ A(z)=2(z+z^3),\\
	T(z)&=2z^5U(z) + 2z^5 + 2z^6, \ \ U(z)=z^2U(z) + T(z) + 2(z^2+z^3).
	\end{align*}
	Solving these yields the following rational expression:
	\begin{equation}\label{abelianseriestwo}
	S(z)=\frac{1+2z-z^2-2z^5-2z^6+2z^7-2z^8}{1-z^2-2z^5}.
	\end{equation}
\end{proof}

\begin{corollary}\label{cor:2abelian}
	The conjugacy classes in $\Z_k$ have growth rate approximately $1.348$.
\end{corollary}
\begin{proof}
	The only real root of the polynomial $1-z^2-2z^5$, the denominator of (\ref{abelianseriestwo}), is approximately $0.742$. Denote this root by $a$, so that $1-a^2-2a^5=0$. Using this identity, we find that the numerator of \ref{abelianseriestwo} is equal to $a+a^2-a^4+a^6$ when $z=a$. Since $a^4<a^2$, we see that $a$ is not a root of the numerator. Therefore the growth rate is the reciprocal of $a$, approximately $1.348$.
\end{proof}

\section{The conjugacy classes $[(x,m)]$, $m\neq 0$, in $BS(1,k)$}\label{sec:general}

In this section we find and describe a set of minimal representatives for the conjugacy classes of the form $(x,m)$ with $m\neq 0$. 
\subsection{The conjugacy geodesics}

We first need the following result by Collins-Edjvet-Gill, which although stated for $k$ even, also holds for $k$ odd.

\begin{lemma}{\cite[Lemma 2.2]{CEG}}\label{geo-lemma}
Let $w$ be a geodesic word. Then:
\begin{enumerate}
\item If $w$ has a subword of the form $t^{-r}a^{i_0}ta^{i_1}t\cdots ta^{i_n}t^{-s}$, where $i_0, i_n \neq 0$, $r,s,n \ge 1$, then $r+s \le n$.
\item If $w$ has a subword of the form $t^{r}a^{i_0}t^{-1}a^{i_1}t^{-1}\cdots t^{-1}a^{i_n}t^{s}$, where $i_0, i_n \neq 0$, $r,s,n \ge 1$, then $r+s \le n$.
\item $w$ has at most one subword of the form $t^{-1}a^{i}t$ where $i \ne 0$, and at most one subword of the form $ta^it^{-1}$ where $i \ne 0$.
\end{enumerate}
\end{lemma}

The following proposition shows that a conjugacy geodesic $w$ has no `pinches', that is, no subwords of the form $t^{-1}a^{i}t$ or $ta^it^{-1}$ where $i \ne 0$.
\begin{proposition}
Every conjugacy geodesic $w$ for $[(x,m)]$ with $m > 0$ must be, up to a cyclic permutation, of the form $a^{x_0}ta^{x_1}t \cdots a^{x_{m-1}}t$ for some $x_0,\cdots,x_{m-1} \in \Z$.
\end{proposition}

\begin{proof}
Let $w$ be a conjugacy geodesic for $[(x,m)]$. 

Suppose that $w$ contains $t^{-1}$ non-trivially. By Lemma \ref{geo-lemma} (3), after cyclically permuting $w$ if necessary, we may assume that \[w = a^{x_0}ta^{x_1}t\cdots a^{x_{n-1}}ta^{y_{n}}t^{-1}a^{y_{n-1}}t^{-1}\cdots a^{y_{m+1}}t^{-1}\] with $x_0, y_n \neq 0$, and $n > m$.
Since $a$ commutes with any word with $t$-exponent sum equal to zero, we can rewrite $w$ as follows, without increasing its length:
\begin{align*}
w&= a^{x_0}ta^{x_1}t\cdots ta^{x_{n-1}}(ta^{y_n}t^{-1})a^{y_{n-1}}t^{-1}a^{y_{n-2}}\cdots a^{y_{m+1}}t^{-1}\\
&= a^{x_0}ta^{x_1}t\cdots a^{x_{n-2}}(t^2a^{y_n}t^{-1}a^{x_{n-1}+y_{n-1}}t^{-1})a^{y_{n-2}}\cdots a^{y_{m+1}}t^{-1}\\
&~~\vdots \\
&= a^{x_0}ta^{x_1}t\cdots a^{x_m} t^{n-m} a^{y_n}t^{-1}a^{x_{n-1}+y_{n-1}}t^{-1}\cdots a^{x_{m+1}+y_{m+1}}t^{-1}.
\end{align*}
 For ease of notation we will rename exponents so that \[w = a^{x_0}ta^{x_1}t\cdots a^{x_m} t^{n-m} a^{y_n}t^{-1}a^{y_{n-1}}t^{-1}\cdots a^{y_{m+1}}t^{-1},\] and note that its cyclic permutation $a^{x_1}t\cdots a^{x_m}t^{n-m}a^{y_n}t^{-1}a^{y_{n-1}}t^{-1}\cdots a^{y_{m+1}}t^{-1}a^{x_0}t$ has a subword $t^{n-m}a^{y_n}t^{-1}a^{y_{n-1}}t^{-1}\cdots a^{y_{m+1}}t^{-1}a^{x_0}t$ which contradicts Lemma \ref{geo-lemma} (2). So, $w$ cannot contain any $t^{-1}$.

Thus, $w$ must have the form $a^{x_0}ta^{x_1}t \cdots a^{x_{m-1}}t$, up to a cyclic permutation.
\end{proof}

Now by checking through the list of geodesics in \cite[Section 4]{CEG}, we see a conjugacy geodesic must be of the form (MWe1a). Translating this to our language and using the fact that a cyclic permutation of a conjugacy geodesic is still a geodesic, we obtain the following proposition:

\begin{proposition}\label{conj-geo}
In $BS(1,k)$, every conjugacy geodesic $w$ for $[(x,m)]$ with $m > 0$ must be, up to a cyclic permutation, of the form $a^{x_0}ta^{x_1}t \cdots a^{x_{m-1}}t$ for some $x_0,\cdots,x_{m-1} \in \Z$ such that:
\begin{itemize}
\item If $k = 2r+1$ is odd, then $|x_i| \le r$ for every $i$.
\item If $k = 2r$ is even, then $|x_i| \le r$, and for each $i$, if $x_{i-1} = r$ then $0 \le x_i < r$, and if $x_{i-1} = -r$ then $-r < x_i \le 0$. (Here and henceforth in this section, we use the convention that $x_{-1} = x_{m-1}$.)
\end{itemize}
\end{proposition}

\subsection{The conjugacy representatives}\label{sec:ConjClasses}

We now give conjugacy representatives for a fixed $m > 0$. Recall that by (\ref{basicconjugation}) two elements $(x,m)$ and $(y,n)$ are conjugate only if $m=n$, so it suffices to restrict the analysis to elements of the form $(x,m)$, with $m$ fixed, in the following arguments.

\begin{lemma}
Suppose $k = 2r+1$ and $m > 0$. Let \[\A_m = \{a^{x_0}ta^{x_1}t \cdots a^{x_{m-1}}t \mid |x_i| \le r\} \setminus \{(a^{-r}t)^m\}.\] Then two words in $\A_m$ are conjugate if and only if they are cyclic permutations of each other, and every word in $\A_m$ is a conjugacy geodesic.
\end{lemma}

The proof of the odd case is the same as the proof for even case, but simpler. Thus, we shall only prove the even case.

\begin{lemma}
Suppose $k = 2r$ and $m > 0$. Let $\A_m$ be the set of words $a^{x_0}ta^{x_1}t \cdots a^{x_{m-1}}t$ satisfying
\begin{enumerate}
\item $|x_i| \le r$,
\item for each $i$, if $x_{i-1} = r$ then $0 \le x_i < r$, and if $x_{i-1} = -r$ then $-r < x_i \le 0$,
\item if $m$ is even, $(a^{-(r-1)}ta^{-r}t)^{\frac m2}$ and $(a^{-r}ta^{-(r-1)}t)^{\frac m2}$ are excluded from $\A_m$.
\end{enumerate}
Then two words in $\A_m$ are conjugate if and only if they are cyclic permutations of each other, and every word in $\A_m$ is a conjugacy geodesic.
\end{lemma}

\begin{proof}
We first show that two distinct words $a^{x_0}ta^{x_1}t \cdots a^{x_{m-1}}t$ and $a^{y_0}ta^{y_1}t \cdots a^{y_{m-1}}t$ in $\A_m$ cannot be conjugate by $a^\ell$, $\ell \neq 0$. Suppose we have such a pair, and suppose these two words represent the elements $(x,m)$ and $(y,m)$. Since $x = \sum\limits_{i = 0}^{m-1} x_i k^i$, (1) implies that $|x| \le (k^m-1)\frac{k}{2(k-1)}$, and similarly $|y| \le (k^m-1)\frac{k}{2(k-1)}$. The conjugation by $a^{\ell}$ translates into $x = y+\ell(k^m-1)$, which together with the above inequalities forces $|\ell| = 1$. Without loss, we will assume that $\ell=1$. 

Now as $\ell = 1$, $x-y=\sum\limits_{i = 0}^{m-1} (x_i-y_i)k^i = k^m-1 = (k-1)+(k-1)k+\cdots+(k-1)k^{m-1}$. Suppose $x_i - y_i \neq k-1$ for some $i$, and let $i_1$ be the smallest such index. By taking 
\begin{equation}\label{xyequation}
\sum\limits_{i = 0}^{m-1} (x_i-y_i)k^i = (k-1)+(k-1)k+\cdots+(k-1)k^{m-1}
\end{equation}
modulo $k^{i_1+1}$, we must have $x_{i_1}-y_{i_1} = -1$ since all higher terms are $0 \pmod{ k^{i_1+1}}$, all lower terms on both sides cancel, and so $(x_{i_1}-y_{i_1})k^{i_1}\equiv (k-1)k^{i_1} \pmod{ k^{i_1+1}}$ (and $|x_{i_1}-y_{i_1}| \le k$).  By taking equation (\ref{xyequation}) modulo $k^{i_1+2}$, similar computations show that $x_{i_1+1}-y_{i_1+1} \equiv 0 \pmod k$. If $x_{i_1+1}-y_{i_1+1} = 0$, then the same argument implies $x_{i_1+2}-y_{i_1+2} \equiv 0 \pmod k$, etc. Suppose $i_2$ is the first index such that $x_{i_2}-y_{i_2} \neq 0$. 
\begin{itemize}
\item If $x_{i_2}-y_{i_2} = k$, this forces $x_{i_2} = r$ and $y_{i_2} = -r$. By (2), this means $0 \le x_{i_2+1} < r$ and $0 \ge y_{i_2+1} > -r$, so $0 \le x_{i_2+1}-y_{i_2+1} \le k-2$. But equation (\ref{xyequation}) modulo $k^{i_2+2}$ implies $x_{i_2+1}-y_{i_2+1} \equiv -1 \pmod k$, a contradiction.
\item If $x_{i_2}-y_{i_2} = -k$, this forces $x_{i_2} = -r$ and $y_{i_2} = r$. By (2), this means $0 \ge x_{i_2+1} > -r$ and $0 \le y_{i_2+1} < r$, so $0 \ge x_{i_2+1}-y_{i_2+1} \ge -k+2$. But equation (\ref{xyequation}) modulo $k^{i_2+2}$ implies $x_{i_2+1}-y_{i_2+1} \equiv 1 \pmod k$, a contradiction.
\end{itemize}
This means that we actually have $x_i - y_i = k-1 = 2r-1$ for all $i$. Thus, $(x_i,y_i) = (r,-r+1)$ or $(r-1,-r)$ for every $i$. But by (2), $m$ must be even since the $r$ and $r-1$ need to alternate in $w$ as a cyclic word, and $a^{y_0}ta^{y_1}t \cdots a^{y_{m-1}}t = (a^{-r}ta^{-r+1}t)^\frac m2$ or $(a^{-r+1}ta^{-r}t)^\frac m2$. This violates (3), and thus any two distinct words in $\A_m$ cannot be conjugate by $a^{\ell}$.

Now let $w = a^{x_0}ta^{x_1}t \cdots a^{x_{m-1}}t \in \A_m$ and suppose some word $u$ in $\A_m$ is conjugate to $w$, that is, $w=u^{(x,l)}$, where $(x,l)\in BS(1,k)$. Consider the cyclic permutation $w'$ of $w$ ending in $t$ given by $w'=w^v$, where $v=a^{x_0}ta^{x_1}t \cdots a^{x_{l-1}}t$ and $x_{p} = x_{p \mod m}$, for any $p \in \N$. Clearly $w'$ is also in $\A_m$, and $v$ has the form $(y,l)$.  Then $u^{(x,l)^{-1}(y,l)}=w'$, so $u$ is conjugate to $w'$ by a power of $a$ since $(x,l)^{-1}(y,l)=(z, 0)$ for some $z \in \Z_k$, which gives a contradiction to our previous claim. Thus $u$ must be a cyclic permutation of $w$, proving the first assertion of the lemma. 

Suppose $w = a^{x_0}ta^{x_1}t \cdots a^{x_{m-1}}t \in \A_m$, and take a conjugacy geodesic $u$ of $[w]$. By Proposition \ref{conj-geo}, $u$ is a word in $\A_m$ if it is not excluded by (3) and by the first assertion of this lemma, $w$ is a cyclic permutation of $u$, thus also a conjugacy geodesic. If $u$ is of the form described in (3), then by the proof above, $w = (a^{r-1}ta^rt)^\frac m2$ or $(a^{r}ta^{r-1}t)^\frac m2$ and has the same length as $u$, so is also a conjugacy geodesic.
\end{proof}

The above discussion concerns the case when $m > 0$. The antiautomorphism $(x,m)\mapsto (x,m)^{-1}=(-\frac{x}{k^m},-m)$ provides a bijection between elements of the form $(x,m)$ and those of the form $(y,-m)$. Since $g^{-1}$ has the same length as $g$, and taking inverses preserves conjugacy, the results above translate to the case when $m<0$. Thus, writing $\A_+ = \bigcup\limits_{m>0} \A_m$ and $\A_- = \A_+^{-1}$, we have the following description of conjugacy representatives:

\begin{corollary}\label{cor:conjgeos} The set $\mathcal{A}$, modulo cyclic permutations, gives a set of minimal length conjugacy representatives for the conjugacy classes of the group $BS(1,k)$ that are not in the base group $\Z_k$.

\begin{enumerate}
	\item Let $k=2r+1$, $r\geq 1$. Then $\mathcal{A}=\mathcal{A}_+ \cup \mathcal{A}_-$, where
	\begin{align*}
	\mathcal{A}_+=&\{a^{x_0}ta^{x_1}t\cdots a^{x_{m-1}}t\mid m\geq1, |x_i|\leq r \ \textrm{for} \ 0 \leq i \leq m-1\}\setminus\{(a^{-r}t)^m\mid m\geq1\},\\
	\mathcal{A}_- =&\{t^{-1}a^{x_0}t^{-1}a^{x_1}\cdots t^{-1}a^{x_{m-1}}\mid m\geq1, |x_i|\leq r\}\setminus\{(t^{-1}a^{-r})^m\mid m\geq1\}.
	\end{align*}
	\item Let $k=2r$, $r\geq 1$. Then $\mathcal{A}=\mathcal{A}_+ \cup \mathcal{A}_-$, where
	\begin{align*}
	\mathcal{A}_+=&\{a^{x_0}ta^{x_1}t\cdots a^{x_{m-1}}t\mid m\geq1, |x_i|\leq r, \forall i(x_{i-1} = \pm r \implies 0 \le \pm x_i< r)\}\\ 
	&\setminus\{(a^{-r+1}ta^{-r}t)^\frac m2, (a^{-r}ta^{-r+1}t)^\frac m2\mid m\geq2, m\equiv 0\pmod 2\},\\
	\mathcal{A}_- =&\{t^{-1}a^{x_0}t^{-1} \cdots t^{-1}a^{x_{m-1}}\mid m\geq1, |x_i|\leq r, \forall i(x_{i} = \pm r \implies 0 \le \pm x_{i-1}< r)\}\\
	&\setminus\{(t^{-1}a^{r}t^{-1}a^{r-1})^\frac m2, (t^{-1}a^{r-1}t^{-1}a^r)^m\mid m\geq2, m\equiv 0\pmod 2\}.
	\end{align*}
\end{enumerate}
\end{corollary}
\begin{proof}
We have already shown that the elements of $\mathcal{A}$ are conjugacy geodesics, and unique up to cyclic permutation. It remains to show that every conjugacy class of $BS(1,k)$ not contained in $\Z_k$ has a representative in $\mathcal{A}$.

By the observation above, we only need to show that for $m>0$, every element of the form $(x,m)$ is conjugate to an element represented by a word in $\mathcal{A}_+$. Again, we will only prove this for the more complicated case of $k=2r$.

First, we show that any element of the form $(x,m)$ is conjugate to an element represented by a word of the form $a^{x_0}ta^{x_1}t\cdots a^{x_{m-1}}ta^{x_m}$ (where $x_i\in\Z$). From \cite{CEG}, the element $(x,m)$ has a (geodesic) representative in one of the following forms:
\begin{itemize}
	\item MWe1a: $a^{x_0}ta^{x_1}t\cdots ta^{x_m}$
	\item MWe2a: $t^{-n}a^{x_0}ta^{x_1}t\cdots t a^{x_{m+n}}$, some $1\leq n<m$
	\item MWe3a: $a^{x_0}ta^{x_1}t\cdots ta^{x_{m+n}}t^{-n}$, some $1\leq n<m$
	\item MWe4a: $t^{-l}a^{x_0}ta^{x_1}t\cdots ta^{x_{m+n+l}}t^{-n}$, some $n,l\geq1$, $n+l<m$.
\end{itemize}
Words in MWe1a are already of the required form. Cyclic permutation ensures that words in MWe2a or MWe4a are conjugate to words in MWe3a. Such a word can be expressed as follows: \[a^{x_0}ta^{x_1}t\cdots ta^{x_{m+n}}t^{-n}=a^{x_0}ta^{x_1}t\cdots t a^{\sum_{j=0}^nx_{m+j}k^j}\] by expressing the suffix $a^{x_m}ta^{x_{m+1}}t\cdots a^{x_{m+n}}t^{-n}$ in terms of $a$ only. This is in the required form.

Next, we show that any element of the form $a^{x_0}ta^{x_1}t\cdots a^{x_{m-1}}ta^{x_m}$ is conjugate to an element of the form $a^{y_0}ta^{y_1}t\cdots a^{y_{m-1}}t$, where $|y_i|\leq r$ for all $i$. To see this, consider the following procedure:
\begin{enumerate}
\item Choose $i<m$ such that $|x_i|>r$. Modify the word using the rewrite $a^{\pm(r+1)}t\mapsto a^{\mp(r-1)}ta^{\pm1}$ (which doesn't change the group element). Repeat this step until there is no such $i$.
\item Cyclically permute the (now possibly altered) $a^{x_m}$ to the front of the word (which doesn't change the conjugacy class). If there is now some $i<m$ with $|x_i|>r$, return to step $1$. Otherwise the procedure terminates.
\end{enumerate}
Clearly if this process terminates we will have a word in the desired form. To see that it does indeed terminate, consider the quantity $h:=\sum_{i=0}^m|x_i|$. The rewrite in step $1$, applied to $a^{x_i}$ say, reduces $|x_i|$ by $2$, and modifies $|x_{i+1}|$ by $\pm1$, depending on signs, thus step $1$ always reduces $h$. Step $2$ cannot increase $h$, it either keeps it constant or reduces it, depending on the signs of $x_0$ and $x_m$. Since $h$ can never be negative, the process must terminate.

Finally, we show that any element of the form $a^{y_0}ta^{y_1}t\cdots a^{y_{m-1}}t$, where $|y_i|\leq r$ for all $i$, is conjugate to an element of $\mathcal{A}_+$. Consider the following procedure:
\begin{enumerate}
	\item If there are any $i$s with $x_{i-1}=\pm r$ and $\pm x_i<0$, rewrite the left-most occurrence according to the rule $a^{\pm r}ta^{x_i}t\mapsto a^{\mp r}ta^{x_i\pm1}t$. Repeat this step until there are no such $i$.
	\item If there are any subwords of the form $a^{\pm r}ta^{\pm r}t$, rewrite the left-most such subword to $a^{\mp r}ta^{\mp (r-1)}ta^{\pm1}$. Repeat this step until there are no such subwords.
	\item If the previous steps have resulted in a new $a^{\pm1}$ appearing at the end of the word, cyclically permute it to the front. Return to step $1$.
\end{enumerate}
It is clear that if this process terminates, the new word will either be an element of $\mathcal{A}_+$, or will be in the set $\{(a^{-r+1}ta^{-r}t)^\frac m2, (a^{-r}ta^{-r+1}t)^\frac m2\mid m\geq2, m\equiv 0\pmod 2\}$. In the latter case, the word is conjugate to an element of $\mathcal{A}_+$. This finishes the proof.

To see that the process terminates, note that since we work from left to right, each step will only be repeated a finite number of times before moving onto the next step. Furthermore, working left to right in step $1$ also ensures that no additional candidates for step $2$ are created. Repeating step $2$ any number of times will result in at most one $a^{\pm1}$ appearing at the right hand end of the word. After cyclically permuting, and returning to step $1$, there may be a subword of the form $a^{\pm r}ta^{\pm r}t$ at the start of the word. However, after repeating step $2$ as many times as necessary, any letter appearing at the right hand end of the word will have the same sign at the previous time, and thus when cyclically permuted cannot result in another subword of the form $a^{\pm r}ta^{\pm r}t$. Thus the process will terminate.
\end{proof}

\section{The conjugacy growth series of $BS(1,k)$}\label{sec:growthseries}

In this section we show, in Corollary \ref{cor:mainresult}, that the conjugacy growth series of $BS(1,k)$ with respect to its standard generating set is transcendental. This follows from determining the asymptotics (and transcendental behaviour) of conjugacy growth outside $\Z_k$ in the following proposition.

\begin{proposition}\label{prop:nonabelian}
The generating function for the number of conjugacy classes in $BS(1,k)$ of the form $[(x,m)]$, with $m\neq 0$, is transcendental.
\end{proposition}

\begin{proof} We compute the asymptotics for the number of conjugacy classes of length $n$ in $BS(1,k)$ by finding the growth of the set $\mathcal{A}$ in Corollary \ref{cor:conjgeos}. 

We start with the odd case $k=2r+1$ and apply Corollary \ref{cor:conjgeos} (1). Since there is a length-preserving bijection between $\mathcal{A}_+$ and $\mathcal{A}_-$, it suffices to consider the asymptotics for $\mathcal{A}_+$. Moreover, since the set $\mathcal{N}_o=\{(a^{-r}t)^m\mid m\geq1\}$ being removed has negligible size (there is at most one word in $\mathcal{N}_o$ of length $n$ for fixed $r$ and $n$), it is sufficient to compute the growth of $\mathcal{A}_o:=\{a^{x_0}ta^{x_1}t\cdots a^{x_{m-1}}t \mid m\geq 1, |x_i|\leq r\}$. Let $\mathcal{S}_o:=\{t, at, a^{-1}t, a^2t, a^{-2}t, \dots, a^rt, a^{-r}t\}$. Then $\mathcal{A}_o$ is equal to $\mathcal{S}_o^*$, so the generating function for $\mathcal{A}_o$ is $\mathcal{A}_o(z)= \frac{1}{1- \mathcal{S}_o(z)},$ where 
\begin{equation}\label{Sformula}
\mathcal{S}_o(z)=z+2z^2+ \dots + 2z^{r+1}=z+2z^2\frac{1-z^r}{1-z}
\end{equation}
 is the generating function of $\mathcal{S}_o$ (see Flajolet, Theorem I.1, p. 27). We get 
\begin{equation}\label{nonZseriesOdd}
\mathcal{A}_o(z)=\frac{1}{1-z-2z^2\frac{1-z^r}{1-z}}=\frac{1-z}{1-2z-z^2+2z^{r+2}}.
\end{equation}
The denominator of $\mathcal{A}_o(z)$, that is, the polynomial $p(z)=1-2z-z^2+2z^{r+2}$, satisfies $p(0)=1>0$ and $p(\frac{1}{2})<0$ (and $p(\frac{1}{2})= 0$ for $r=1$), so it has a root $\rho_o \in (0,\frac{1}{2})$ (and $\rho_o=\frac{1}{2}$ for $r=1$). Moreover, $p'(\alpha)=-2-2\alpha+2(r+2)\alpha^{r+1}<0$ for $0<\alpha<\frac{1}{2}$, so $\rho_o$ is a simple root. Also, $1-2z-z^2+2z^{r+2} = (1-z^2)-2z(1-z^{r+1})$, so it has no root in $(-1,0)$. Thus the growth rate of the set $\mathcal{A}_o$ is $\frac{1}{\rho_o}>2$, which implies that the number of words of length $n$ in $\mathcal{A}_o$, and therefore also $\mathcal{A}$, is asymptotically $c_o(r) \rho_o^{-n}$, where $c_o(r)$ is a constant depending on $r$. 

Now let $k$ be even, $k=2r$. The counting is similar, except that we impose on the set $\mathcal{A}_o:=\{a^{x_0}ta^{x_1}t\cdots a^{x_{m-1}}t \mid m\geq 1, |x_i|\leq r\}$ considered above the conditions from Corollary \ref{cor:conjgeos}(2), that is, $a^rt$ and $a^{-r}t$ can each be followed only by $r$ words out of the total $2r+1$ in $\mathcal{S}$. Call the set with these restrictions $\mathcal{A}_e$, and let $\mathcal{S}_e=\{t, at, a^{-1}t, \dots, a^{r-1}t, a^{-r+1}t, a^{\pm r}tt, a^{\pm r}ta^{\pm 1}t, \dots , a^{\pm r}ta^{\pm(r-1)}t\}$ (and $\mathcal{S}_e=\{t, a^{\pm 1}tt\}$ for $r=1$). Note that $\mathcal{S}_e^*$ does not include any words that end in $a^rt$ or $a^{-r}t$, but since we need to consider the set $\mathcal{A}_e$ up to cyclic permutations, the set $\mathcal{S}_e^*$ will in fact suffice to give the asymptotics for $\mathcal{A}_e$ up to cyclic permutations, since it ensures only `legal' occurrences of $a^rt$ or $a^{-r}t$ appear when cyclically permuting the words.

Then since 
\begin{align}\label{Seformula}
\mathcal{S}_e(z)&=z+2z^2+ \dots + 2z^{r}+2z^{r+2}+\dots+2z^{2r+1}\\
\nonumber&=z+2z^2\frac{z^{2r}-1}{z-1}-2z^{r+1}=\frac{-z-z^2+2z^{r+1}-2z^{r+2}+2z^{2r+2}}{z-1}
\end{align}
we have
\begin{align}\label{nonZseriesEven}
\mathcal{S}_e^*&=\frac{1-z}{1-2z-z^2+2z^{r+1}-2z^{r+2}+2z^{2r+2}}.
\end{align}

For $r> 1$, the denominator $p(z)$ of  (\ref{nonZseriesEven}) satisfies $p(0)=1>0$ and $p(\frac{1}{2})<0$, so it has a root $\rho_e \in (0,\frac{1}{2})$. Moreover, $p'(\alpha)=-2-2\alpha+2(r+1)\alpha^{r}-2(r+2)\alpha^{r+1}+2(2r+2)\alpha^{2r+1}<0$ for $0<\alpha<\frac{1}{2}$, so $\rho_e$ is a simple root, and the growth of the languages $\mathcal{S}_e^*$, and consequently $\mathcal{A}_e$, is $\frac{1}{\rho_e}>0$. (For $r = 1$, $\rho_e \approx 0.590$.) Also, $1-2z-z^2+2z^{r+1}-2z^{r+2}+2z^{2r+2} = (1-z^2)-2z(1+z^{r+1})(1-z^r)$, so it has no root in $(-1,0)$. This implies that the number of words of length $n$ in $\mathcal{A}_e$, is asymptotically $c_e(r) \rho_e^{-n}$, where $c_e(r)$ is a constant depending on $r$.

Now in order to find the growth of the conjugacy classes for $m\neq 0$, we need to count the number of representatives of length $n$ in $\mathcal{A}_o$ or $\mathcal{A}_e$, up to the cyclic permutation of the subwords in $\mathcal{S}_o$ or $\mathcal{S}_e$.  For each word in $\mathcal{A}_o$ or $\mathcal{A}_e$ there are $m$ possible distinct cyclic permutations unless that word is a non-trivial power. Given that the number of powers is negligible compared to the total number of words, for fixed $n$ and $m$ the number of cyclic representatives of words in $\mathcal{A}_o$ and $\mathcal{A}_e$ is approximately $c_o(r)\frac{\rho_o^{-n}}{m}$ and $c_e(r)\frac{\rho_e^{-n}}{m}$, respectively. Since each word of length $n$ in $\mathcal{A}_o$ or $\mathcal{A}_e$ consists of $m$ `syllables' of bounded length we get  $\frac{n}{r+1} \leq m \leq n$ in the odd case and  $\frac{n}{r+\frac12} \leq m \leq n$ in the even case, so the number $a^o_n$ of cyclic representatives in the odd case satisfies 
\begin{equation}\label{inequality:odd}
c_o(r)\frac{\rho_o^{-n}}{n} \leq a^o_n\leq c_o(r)\frac{(r+1)\rho_o^{-n}}{n}
\end{equation}
 and in the even case the number $a^e_n$ of cyclic representatives satisfies 
 \begin{equation}\label{inequality:even}
 c_e(r)\frac{\rho_e^{-n}}{n} \leq a^e_n\leq c_e(r)\frac{(r+\frac12)\rho_e^{-n}}{n}
 \end{equation}

%



Finally, by \cite[Theorem D]{Flajolet87} the generating function for any sequence with asymptotics of the form (\ref{inequality:odd}) or (\ref{inequality:even}), that is, bounded on both sides by terms $\frac{\rho^n}{n}$(up to multiplicative constants), is transcendendal.
\end{proof}

\begin{corollary} \label{cor:mainresult}
The conjugacy growth series for $BS(1, k)$, with respect to the generating set $\{a, t\}$, is transcendental.
\end{corollary}
\begin{proof}

By Propositions \ref{prop:abelianodd}, \ref{prop:abelianeven}, \ref{prop:abelian2}, the conjugacy growth series for $\Z_k$ (when $m=0$) is rational, and by Proposition \ref{prop:nonabelian} the generating function for conjugacy classes of the form $[(x,m)]$ with $m\neq 0$ is transcendental. Since the sum of a transcendental function and a rational function is transcendental, we obtain the result.
%
\end{proof}

\begin{corollary}\label{cor:equalrates}
The conjugacy and standard growth rates of $BS(1, k)$, with respect to the generating set $\{a, t\}$, are equal.
\end{corollary}

\begin{proof}
We start with the odd case. By \cite[Theorem (iii)]{CEG} (see also \cite[Lemma 11(b)]{BT}) the standard growth rate is the inverse of the smallest absolute value of the real roots of the polynomial $1-2t-t^2+2t^{r+2}$ which appears in the denominator of the standard growth series. But the same polynomial appears in the denominator of (\ref{nonZseriesOdd}), and since the smallest absolute value of real roots is $\rho_o\le\frac12$, this will dominate the growth rate of the conjugacy classes in $\Z_k$, which is smaller than $2$ by Corollary \ref{cor:oddabelian}. Thus the standard and the conjugacy growth rates are equal. 

In the even case with $k > 2$, note that the second factor in the denominator in \cite[Theorem (i)]{CEG} is identical to that in formula (\ref{nonZseriesEven}), and both denominators have the same smallest absolute value of real roots $\rho_e<\frac12$ which dominates the growth rate of the conjugacy classes in $\Z_k$, which is smaller than $2$ by Corollaries \ref{cor:oddabelian}, so the two rates are equal.

In the case when $k = 2$, note that the second factor in the denominator in \cite[Theorem (ii)]{CEG} is also a factor to that in formula (\ref{nonZseriesEven}), and both denominators have the same smallest absolute value of real roots $\rho_e \approx 0.590$ which dominates the growth rate of the conjugacy classes in $\Z_k$, which is approximately $\frac1{0.742} \approx 1.348$ by Corollaries \ref{cor:2abelian}, so the two rates are equal.
\end{proof}

\section{Conjugacy growth series formulas} \label{sec:formulas}

In this section we give formulas for the growth series of the conjugacy classes of $BS(1,k)$ outside the normal abelian subgroup $\Z_k$. That is, we compute the generating function for the set $\mathcal{A}$, up to cyclic permutation, given in Corollary \ref{cor:conjgeos}.

In the description of $\mathcal{A}$ in Corollary \ref{cor:conjgeos} there is a length-preserving bijection between $\mathcal{A}_+$ and $\mathcal{A}_-$, so it suffices to consider the generating function for the set $\mathcal{A}_+$ up to cyclic permutations. 

In the odd $k=2r+1$ case, as the set $\mathcal{N}_o=\{(a^{-r}t)^m\mid m\geq1\}$ has generating function $\mathcal{N}_o(z)=\sum_{m\geq 1}z^{(r+1)m}$, it is sufficient to compute the generating function of $\mathcal{A}_o:=\{a^{x_0}ta^{x_1}t\cdots a^{x_{m-1}}t \mid m\geq 1, |x_i|\leq r\}$ up to cyclic permutation. 

In the $k=2r$ case as the set $\mathcal{N}_e=\{(a^{-r+1}ta^{-r}t)^\frac m2, (a^{-r}ta^{-r+1}t)^\frac m2\mid m\geq2, m\equiv 0\pmod 2\}$ has generating function $\mathcal{N}_o(z)=\sum_{m\geq 1}z^{(2r+1)m}$, it is sufficient to compute the generating function of $\mathcal{A}_e=\{a^{x_0}ta^{x_1}t\cdots a^{x_{m-1}}t\mid m\geq1, |x_i|\leq r, \forall i(x_{i-1} = \pm r \implies 0 \le \pm x_i< r)\}$ up to cyclic permutation.

This is exactly the \textit{cycle construction} (see page 26 in \cite{FlajoletSedgewick}) applied to the sets $$\mathcal{S}_o=\{t, at, a^{-1}t, a^2t, a^{-2}t, \dots, a^rt, a^{-r}t\}$$ and $$\mathcal{S}_e=\{t, at, a^{-1}t, \dots, a^{r-1}t, a^{-r+1}t, a^{\pm r}tt, a^{\pm r}ta^{\pm 1}t, \dots , a^{\pm r}ta^{\pm(r-1)}t\},$$ respectively, defined in the proof of Proposition \ref{prop:nonabelian}. Thus by applying the formula in \cite[Theorem I.1]{FlajoletSedgewick}, we get that 
\begin{equation}
Cyc(\mathcal{A}_o)(z)= \sum_{k=1}^\infty \frac{-\phi(k)}{k}\log(1-\mathcal{S}_o(z^k)),
\end{equation}
where $\mathcal{S}_o(z)$ is given in (\ref{Sformula}), and in the odd case we get
 \begin{equation}
Cyc(\mathcal{A}_e)(z)= \sum_{k=1}^\infty \frac{-\phi(k)}{k}\log(1-\mathcal{S}_e(z^k)),
\end{equation}
where $\mathcal{S}_e(z)$ is given in (\ref{Seformula}).

The conjugacy growth series for $BS(1, 2r+1)$ is then the series obtained by adding (\ref{abelianseriesodd}) to $Cyc(\mathcal{A}_o)$ and then subtracting $\mathcal{N}_o(z)$, and  the conjugacy growth series for $BS(1, 2r)$ is then the series obtained by adding (\ref{abelianserieseven}) to $Cyc(\mathcal{A}_e)$ and then subtracting $\mathcal{N}_e(z)$.

\section{Conjectures and open questions}

While this paper establishes qualitative and quantitative results for conjugacy growth in $BS(1,k)$ with respect to the standard generating set, we conjecture that the same characterisations of conjugacy growth should hold for all generating sets. More generally, we expect the following to be true. Clearly the second conjecture implies the first. 

\begin{conjecture}
The conjugacy growth series of the groups $BS(1,k)$ with respect to \emph{any} generating set are transcendental.
\end{conjecture}

\begin{conjecture}[see also \cite{Evetts}]
The conjugacy growth series of any finitely presented group that is not virtually abelian is transcendental.
\end{conjecture}

Regarding growth rates, we ask the following question.

\begin{question}\label{conj:equal1}
If the conjugacy and standard growth rate of a group are equal for some generating set, are they equal for all generating sets?
\end{question}

The question is related to the conjecture below, which, as we pointed out in the introduction, holds in many important classes of groups.

\begin{conjecture}\label{conj:equalmany}
For any choice of generating set, the conjugacy and standard growth rate of a finitely presented group are equal.
\end{conjecture}

\section*{Acknowledgements}

The authors acknowledge the hospitality of the Hausdorff Institute in Bonn and thank the organisers of the Trimester \emph{Logic and Algorithms in Group Theory}, where the discussions on this paper started.
 The first named author was partially supported by EPSRC Standard grant EP/R035814/1.


\end{document}